\documentclass{amsart}

\usepackage[active]{srcltx}

\usepackage{graphicx}

\usepackage{latexsym}
\usepackage{amsmath,amsthm}
\usepackage{amsfonts}
\usepackage[psamsfonts]{amssymb}
\usepackage{enumerate}

\usepackage[final]
{showkeys}
\usepackage{url}

\RequirePackage[OT1]{fontenc}

\usepackage{tikz}
\usetikzlibrary{decorations}
\usetikzlibrary{decorations.pathreplacing}

\usepackage{hyperref}

\theoremstyle{plain} 
\newtheorem{theorem}{Theorem}%[section]
\newtheorem{corollary}[theorem]{Corollary}
\newtheorem{lemma}[theorem]{Lemma}

\newtheorem*{conjecture*}{Conjecture}
\theoremstyle{definition} 

\theoremstyle{definition} 

\newtheorem{remark}[theorem]{Remark}
\newtheorem*{remark*}{Remark}
\makeatletter
  \renewcommand\section{\@startsection {section}{1}{\z@}%
                                   {-\bigskipamount}%
                                   {\medskipamount}%
                                   {\large\bfseries%\mathversion{bold}
                                   \raggedright}}

  \renewcommand\subsection{\@startsection {subsection}{2}{\z@}%
                                   {-\medskipamount}%
                                   {\smallskipamount}%
                                   {\bfseries%\mathversion{bold}
                                   \raggedright}}
\makeatother

\newcommand{\card}{\operatorname{card}}

\renewcommand{\gg}{>\kern-2pt>}
\renewcommand{\ll}{<\kern-2pt<}

\renewcommand{\gg}{>\kern-2pt>}
\renewcommand{\ll}{<\kern-2pt<}

\renewcommand{\le}{\leqslant}
\renewcommand{\ge}{\geqslant}

\newcommand{\1}{\mathbf 1}

\newcommand{\al}{\alpha}
\newcommand{\be}{\beta}
\newcommand{\ga}{\gamma}

\renewcommand{\th}{\omega}
\newcommand{\thh}{\theta}

\newcommand{\si}{\sigma}
\newcommand{\Si}{\Sigma}
\newcommand{\la}{\lambda}

\newcommand{\de}{\delta}
\newcommand{\De}{\Delta}

\newcommand{\D}{\mathcal{D}}
\newcommand{\EE}{\mathcal{E}}
\newcommand{\T}{\operatorname{\mathcal{T}}}

\newcommand{\LD}{\mathcal{L}\!\mathcal{D}}
\renewcommand{\LD}{\mathcal{L}{\kern -1.9pt}\mathcal{D}}
\renewcommand{\LD}{\mathcal{D}}
\renewcommand{\LD}{\mathcal{L}{\kern -4pt}\mathcal{C}}
\renewcommand{\LD}{\mathcal{R}{\kern -3pt}\mathcal{C}}

\renewcommand{\P}{\operatorname{\mathsf{P}}} 

\newcommand{\E}{\operatorname{\mathsf{E}}}

\newcommand{\Cov}{\operatorname{\mathsf{Cov}}}

\newcommand{\R}{{\mathbb{R}}}
\newcommand{\C}{\mathbb{C}}
\renewcommand{\C}{\mathcal{C}}

\newcommand{\tDe}{\tilde\De}
\newcommand{\tT}{\tilde\T}

\renewcommand{\D}{\mathrel{\overset{\operatorname{D}}=}}

\newcommand{\inter}{\mathrm{int}\,}

\begin{document}

%\begin{frontmatter}

\title{Excess Versions of the Minkowski and H\"older Inequalities}

%% use optional labels to link authors explicitly to addresses:
%% \author[label1,label2]{}
%% \address[label1]{}
%% \address[label2]{}

\author{Iosif Pinelis}

\address{Department of Mathematical Sciences\\
Michigan Technological University\\
Hough\-ton, Michigan 49931, USA\\
E-mail: ipinelis@mtu.edu}

\keywords{Minkowski's inequality, H\"older's inequality, $p$-excess}

\subjclass[2010]{26D15, 60E15}

%42A85 View Publications (1980-now) Convolution, factorization 
%
% 	60E07   	Infinitely divisible distributions; stable distributions
%		60E10   	Characteristic functions; other transforms 
%		
%		62E10 View Publications (1973-now) Characterization and structure theory 

\begin{abstract}
Certain excess versions of the Minkowski and H\"older inequalities are given. These new results generalize and improve the Minkowski and H\"older inequalities. 
\end{abstract}

\maketitle

%\tableofcontents

%\tableofcontents

%\section{Introduction, 
%summary, 
%and discussion
%}\label{intro} 

\section{Introduction and summary}\label{intro}

Let $p$ and $q$ be positive real numbers such that $\frac1p+\frac1q=1$; then, of course, $p>1$ and $q>1$. Let $X$ and $Y$ denote nonnegative random variables (r.v.'s), defined on the same probability space. Then one has the Minkowski inequality 
\begin{equation*}
	\|X+Y\|_p\le\|X\|_p+\|Y\|_p
\end{equation*}
and the H\"older inequality 
\begin{equation*}
	\E XY\le\|X\|_p\|Y\|_q, 
\end{equation*}
where, as usual, $\|X\|_p:=\E^{1/p}|X|^p$; see, e.g., \cite{royden}. From now on, to avoid unpleasant trivialities, let us assume that $\|X\|_p+\|Y\|_p+\|Y\|_q<\infty$. 

%The H\"older inequality with $Y=1$ implies $\|X\|_1\le\|X\|_p$. 

A special case of H\"older's inequality is Lyapunov's inequality, which states that $\E X^\al$ is log-convex in real $\al$, with the conventions $0^0:=1$, $0^\al:=\infty$ for $\al<0$, 
and $0\cdot\infty:=0$, %and $\ln\infty:=\infty$, 
so that $\E X^0=1$, and $\E X^\al=\infty$ if $\al<0$ and $\P(X=0)>0$. In particular, we have $\|X\|_1\le\|X\|_p$. 

So, we may define the (always nonnegative) $p$-excess of $X$ by the formula 
\begin{equation*}
	\EE_p(X):=\big(\|X\|_p^p-\|X\|_1^p)^{1/p}. 
\end{equation*}
One may note that $\EE_2(X)$ is the standard deviation of the r.v.\ $X$. 
Introduce also the covariance-like expression 
\begin{equation*}
	\C_p(X,Y):=\E X^{p-1}Y-\E^{p-1}X\,\E Y,  
\end{equation*}
which is the true covariance, $\Cov(X,Y)$, of the r.v.'s $X$ and $Y$ in the case $p=2$. 

As will be shown in this note, the following Minkowski-like and H\"older-like inequalities for the $p$-excess hold: if $p\le2$ (so that $1<p\le2$), then 
\begin{equation}\label{eq:1}
	\EE_p(X+Y)\le\EE_p(X)+\EE_p(Y)  
\end{equation}
and 
\begin{equation}\label{eq:2}
	\C_p(X,Y)\le\EE_p(X)^{p-1}\EE_p(Y). 
\end{equation}
In the case $p=2$ inequality \eqref{eq:2} becomes the covariance inequality, that is, the Cauchy--Schwarz inequality for the centered r.v.'s $X-\E X$ and $Y-\E Y$.

More generally, for $\thh\in[0,1]$ define the $(p,\thh)$-excess of $X$ by the formula 
\begin{equation*}
	\EE_{p,\thh}(X):=\big(\|X\|_p^p-\thh^p\|X\|_1^p)^{1/p},  
\end{equation*}
which interpolates between $\|X\|_p=\EE_{p,0}(X)$ and $\EE_p(X)=\EE_{p,1}(X)$, and then also  
\begin{equation*}
	\C_{p,\thh}(X,Y):=\E X^{p-1}Y-\thh^p\E^{p-1}X\,\E Y,   
\end{equation*}
which interpolates between $\Cov(X^{p-1},Y)=\C_{p,0}(X,Y)$ and $\C_p(X)=\C_{p,1}(X,Y)$. 

Inequalities \eqref{eq:1} and \eqref{eq:2}, along with the Minkowski and H\"older inequalities, can be extended as follows: 

\begin{theorem}\label{th:1}
Suppose that $p\le2$ (so that $1<p\le2$). Then for all $\thh\in[0,1]$ 
\begin{equation}\label{eq:1st}
	\EE_{p,\thh}(X+Y)\le\EE_{p,\thh}(X)+\EE_{p,\thh}(Y)
\end{equation}
and 
\begin{equation}\label{eq:2nd}
	\C_{p,\thh}(X,Y)\le\EE_{p,\thh}(X)^{p-1}\EE_{p,\thh}(Y). 
\end{equation}
For any real $p>2$ and any $\thh\in(0,1]$, inequalities \eqref{eq:1st} and \eqref{eq:2nd} do not hold in general. 
\end{theorem}

Obviously, the Minkowski and H\"older inequalities are the special cases of 
inequalities \eqref{eq:1st} and \eqref{eq:2nd}, respectively, corresponding to $\thh=0$, and \eqref{eq:1} and \eqref{eq:2} are the special cases of 
\eqref{eq:1st} and \eqref{eq:2nd} corresponding to $\thh=1$. 
Moreover, 
considerations in Round~\ref{th=1} of the proof of \eqref{eq:2nd}, to be given in Section~\ref{proof}, 
show that inequality \eqref{eq:2nd} is, in a sense, an improvement of H\"older's inequality (for $p\in(1,2)$). 
Similarly, the derivation of \eqref{eq:1st} from \eqref{eq:2nd} in the paragraph containing formulas \eqref{eq:g} and \eqref{eq:g'} shows that inequality \eqref{eq:1st} is an improvement of Minkowski's inequality (again for $p\in(1,2)$). 

%However, the assumption, made in the beginning of this paper, that the r.v.'s $X$ and $Y$ be nonnegative is essential. 

\section{Proof of Theorem~1%\protect{\ref{th:1}}
}\label{proof}
%\section{Proof of Theorem~\texorpdfstring{\ref{th:1}}{}}\label{proof}

We shall see at the end of this section that inequalities \eqref{eq:1st} and \eqref{eq:2nd} are easy to obtain from each other, so that it is enough to prove one of them.

\begin{proof}[Proof of %Theorem~\ref{th:1}
inequality \eqref{eq:2nd}]

%!!! philosophy on the number of variables, structure of extremal problems
This proof is much more difficult than that of H\"older's inequality. It will be done by a number of rounds of reduction of the difficulty of the problem. 

\smallskip

%\newcounter{roundno}
%\setcounter{roundno}{0}
%\newcommand{\round}[1]{%\stepcounter{roundno}
%\arabic{\refstepcounter{roundno}\label{#1}}}

\newcounter{theround} \setcounter{theround}{0}
\newcommand{\round}[1]{\noindent%
	\refstepcounter{theround}\textbf{Round \arabic{theround}: #1\!\quad}}%

\round{Reduction to the case $\thh=1$ \label{th=1}}
%
%\noindent{\bf Round 1: Reduction to the case $\thh=1$.} 
Consider the differences 
\begin{equation}\label{eq:De_p,th}
	\De_{p,\thh}(X,Y):=\C_{p,\thh}(X,Y)-\EE_{p,\thh}(X)^{p-1}\EE_{p,\thh}(Y) 
\end{equation}
and 
\begin{equation}\label{eq:De_p}
	\De_p(X,Y):=\De_{p,1}(X,Y):=\C_p(X,Y)-\EE_p(X)^{p-1}\EE_p(Y)
\end{equation}
between the left and right sides of inequalities \eqref{eq:2nd} and \eqref{eq:2}, respectively. For nonnegative real numbers $A,B,C$, consider also 
\begin{equation}\label{eq:De=..aBC}
	\begin{aligned}
	\De_{p;A,B,C}(X,Y)=&A+\E X^{p-1}Y-\E^{p-1} X\,\E Y \\ 
	&-\big(B+\E X^p-\E^p X\big)^{1/q}\,\big(C+\E Y^p-\E^p Y\big)^{1/p}.    
\end{aligned}	
\end{equation} 

The following lemma will also be used in Round~\ref{mass infty} of this proof.

\begin{lemma}\label{lem:A,B,C} 
Suppose that the nonnegative real numbers $A,B,C$ are such that $A\le B^{1/q}C^{1/p}$. Then 
$\De_{p;A,B,C}(X,Y)\le\De_p(X,Y)$. 
\end{lemma}

\begin{proof}
%First here, note that $\De_{p:0,0,0}(X,Y)=\De_p(X,Y)$. Next, 
Since $\De_{p;A,B,C}(X,Y)$ is nondecreasing in $A$, without loss of generality (wlog) $A=B^{1/q}C^{1/p}$. If $B=0$ or $C=0$, then $A=0$, and so, the inequality $\De_{p;A,B,C}(X,Y)\break
\le\De_p(X,Y)$ is trivial. Hence, wlog $B>0$ and $C>0$, and then we can write 
$C=\ga^p B$ and $A=\ga B$ for some real $\ga>0$. 
Let now 
\begin{equation}\label{eq:d}
d(B):=\De_{p;\ga B,B,\ga^p B}(X,Y). 	
\end{equation}
Introduce also 
\begin{equation*}
	c:=\big(\ga^p B+\E Y^p-\E^p Y\big)^{1/p}\big/\big(B+\E X^p-\E^p X\big)^{1/p}
\end{equation*}
and then $a:=\ga c^{-1/q}$ and $b:=c^{1/q}$. Then  
\begin{equation*}
	d'(B)=\ga-\frac1q\,c-\frac1p\,\ga^p\,c^{-p/q}
	=ab-\Big(\frac{a^p}p+\frac{b^q}q\Big)\le0
\end{equation*}
for all $B>0$, by Young's inequality. So, $\De_{p;A,B,C}(X,Y)=\De_{p;\ga B,B,\ga^p B}(X,Y)=d(B)\le d(0)=\De_{p:0,0,0}(X,Y)=\De_p(X,Y)$. 
Lemma~\ref{lem:A,B,C} is thus proved. 
\end{proof}

Now take any $\thh\in[0,1]$ and note that $\De_{p,\thh}(X,Y)=\De_{p;A,B,C}(\thh X,\thh Y)$ with $A:=(1-\thh^p)\E X^{p-1}Y$, $B:=(1-\thh^p)\E X^p$, and $C:=(1-\thh^p)\E Y^p$, so that, by H\"older's inequality, the condition $A\le B^{1/q}C^{1/p}$ of Lemma~\ref{lem:A,B,C} holds,  
which yields $\De_{p,\thh}(X,Y)\le\De_p(\thh X,\thh Y)$. Thus, to prove inequality \eqref{eq:2nd}, it is enough to prove its special case, inequality \eqref{eq:2}.

\smallskip 

%\subsection{Round 1: Removing the case $p=2$.}
%\noindent{\bf Round 1: Removing the case $p=2$.}
\round{Removing the case $p=2$ \label{p ne2}}
% Let us begin this proof by a few preliminary remarks. First, as 
This round is very easy. 
As was noted, the case $p=2$ of \eqref{eq:2} is the Cauchy--Schwarz inequality. So, it is enough to prove \eqref{eq:2} for $p\in(1,2)$, which will be henceforth assumed. %??? needed?

\smallskip 

\round{``Finitization'' of the probability space \label{finit}}
%\noindent{\bf Round 2: ``Finitization'' of the probability space.}
%Next, 
%%by homogeneity, without loss of generality (wlog) $\E X=1$. 
%%Further, 
%w
Wlog  
the r.v.'s $X$ and $Y$ take only finitely many values 
(one may approximate $X$ and $Y$ from below by nonnegative simple r.v.'s and then use the monotone convergence theorem). Therefore, wlog $X$ and $Y$ are defined on a finite probability space. For instance, we may assume that the probability space is $(I,\Si,\mu)$, where $I$ is the finite set $\{(x,y)\colon\P(X=x,Y=y)>0\}$, $\Si$ is the $\si$-algebra of all subsets of $I$, the probability measure $\mu$ is defined by the condition $\mu(\{i\})=w_i:=\P(X=x,Y=y)$ for all $i=(x,y)\in I$, and the r.v.'s $X$ and $Y$ are defined by the conditions $X(i)=x$ and $Y(i)=y$ for all $i=(x,y)\in I$. 
So, the r.v.'s $X$ and $Y$ maybe identified with finite-dimension vectors $(x_i)_{i\in I}$ and $(y_i)_{i\in I}$, respectively. 

\smallskip 

\round{Reduction to an extremal problem \label{extr}}
%\noindent{\bf Round 3: Reduction to an extremal problem.}
Introducing also the vector $W:=(w_i)_{i\in I}$, we can rewrite inequality 
\eqref{eq:2} as 
\begin{equation}\label{eq:?}
\sup\big\{\De_p(X,Y,W)\colon(X,Y,W)\in\T_{I;m_{1,1},m_{1,p},m_{2,1},m_{2,p}}\big\} \overset{\text(?)}\le0,
\end{equation}
where $m_{1,1},m_{1,p},m_{2,1},m_{2,p}$ are any (strictly) positive real numbers, 
\begin{multline*}
	\De_p(X,Y,W):=(X^{p-1}Y)\cdot W-(X\cdot W)^{p-1}\,Y\cdot W \\ 
	-\big(X^p\cdot W-(X\cdot W)^p\big)^{1/q}\,\big(Y^p\cdot W-(Y\cdot W)^p\big)^{1/p}, 
\end{multline*}
the symbol $\cdot$ denotes the dot product in $\R^I$, and $\T_{I;m_{1,1},m_{1,p},m_{2,1},m_{2,p}}$ is the set of all triples $(X,Y,W)$ of vectors $X=(x_i)_{i\in I}$, $Y=(y_i)_{i\in I}$, and $W=(w_i)_{i\in I}$ with nonnegative coordinates such that 
\begin{gather*}
\1\cdot W=\sum_{i\in I}w_i=1,\quad X^p\cdot W=\sum_{i\in I}x_i^p w_i=m_{1,p}, \quad  
Y^p\cdot W=\sum_{i\in I}y_i^p w_i=m_{2,p},\\ 
X\cdot W=\sum_{i\in I}x_i w_i=m_{1,1},\quad 
Y\cdot W=\sum_{i\in I}y_i w_i=m_{2,1};  
\end{gather*}
here and in what follows, $\1:=(1)_{i\in I}$, the vector with all coordinates equal $1$. 
One may note that, in view of the standard convention $\sup\emptyset=-\infty$, inequality \eqref{eq:?} is trivial whenever $m_{1,1},m_{1,p},m_{2,1},m_{2,p}$ are such that  $\T_{I;m_{1,1},m_{1,p},m_{2,1},m_{2,p}}=\emptyset$. 
A reason for the numbers $m_{1,1},m_{1,p},m_{2,1},m_{2,p}$ to be assumed strictly positive is that, if at least one of them is $0$, then for any $(X,Y,W)\in\T_{I;m_{1,1},m_{1,p},m_{2,1},m_{2,p}}$ at least one of the r.v.'s $X,Y$ is almost surely $0$, which makes inequality \eqref{eq:?} trivial. 

\smallskip 

%!!! philosophy 
\round{Compactification, by a change of variables \label{compact}}
%\noindent{\bf Round 4: Compactification, by a change of variables.}
To solve an extremal problem such as the one stated in Round~\ref{extr}, it is natural to use 
%
%What turns out to work in order to solve the stated extremal problem is 
the method of Lagrange multipliers. To be able to do that, we need to ensure a priori that the supremum in \eqref{eq:?} is attained. However, this does not seem easy to do, since the set $\T_{I;m_{1,1},m_{1,p},m_{2,1},m_{2,p}}$ is not bounded and hence not compact in general; 
indeed, for any real $\be>0$ and any $i\in I$ such that $w_i=0$, one may take however large $x_i\ge0$ so that the condition $x_i^\be w_i=0$ hold.
%
% (indeed, note that $x_i^\be w_i=0$ for any $i\in I$, any real $x_i\ge0$, and any real $\be>0$ -- provided that $w_i=0$). 

An appropriate way to compactify the set $\T_{I;m_{1,1},m_{1,p},m_{2,1},m_{2,p}}$ is to use the following new variables: for $i\in I$, let 
\begin{equation}\label{eq:u,v def}
	u_i:=x_i^p w_i\quad\text{and}\quad v_i:=y_i^p w_i,	
\end{equation}
so that 
\begin{equation*}
	x_i w_i=u_i^{1/p}w_i^{1/q},\quad y_i w_i=v_i^{1/p}w_i^{1/q},\quad x_i^{p-1}y_i w_i=u_i^{1/q}v_i^{1/p}.   
\end{equation*}
%where $q$ is dual to $p$, so that $\frac1p+\frac1q=1$. 
Then \eqref{eq:?} will follow from 
\begin{equation}\label{eq:??}
\sup\big\{\tDe_p(U,V,W)\colon(U,V,W)\in\tT_{I;m_{1,1},m_{1,p},m_{2,1},m_{2,p}}\big\} \overset{\text(?)}\le0,
\end{equation}
where $\tT_{I;m_{1,1},m_{1,p},m_{2,1},m_{2,p}}$ is the set of all triples $(U,V,W)$ of vectors $U=(u_i)_{i\in I}$, $V=(v_i)_{i\in I}$, and $W=(w_i)_{i\in I}$ with nonnegative coordinates such that 
\begin{gather}
\1\cdot W=1,\quad U\cdot\1=m_{1,p},\quad V\cdot\1=m_{2,p}, \label{eq:1.W,U.1,V.1} \\ 
U^{1/p}\cdot W^{1/q}=m_{1,1},\quad V^{1/p}\cdot W^{1/q}=m_{2,1} \label{eq:U.W,V.W}
\end{gather}
and, for $(U,V,W)\in\tT_{I;m_{1,1},m_{1,p},m_{2,1},m_{2,p}}$, 
\begin{align}
	\tDe_p(U,V,W):=&U^{1/q}\cdot V^{1/p}-(U^{1/p}\cdot W^{1/q})^{p-1}\,V^{1/p}\cdot W^{1/q} \notag\\ 
	&-\big(U\cdot\1-(U^{1/p}\cdot W^{1/q})^p\big)^{1/q}\,
	\big(V\cdot\1-(V^{1/p}\cdot W^{1/q})^p\big)^{1/p} \label{eq:tDe1} \\  
	 =&U^{1/q}\cdot V^{1/p}-m_{1,1}^{p-1}\,m_{2,1} 
%	 \\ 
%	&
	-\big(m_{1,p}-m_{1,1}^p\big)^{1/q}\,
	\big(m_{2,p}-m_{2,1}^p\big)^{1/p}.   \label{eq:tDe2} 
\end{align} 
% $\1\cdot W=\sum_{i\in I}w_i=1$, $U\cdot\1=\sum_{i\in I}u_i=m_{1,p}$, $V\cdot\1=\sum_{i\in I}v_i=m_{2,p}$, 
%$U^{1/p}\cdot W^{1/q}=\sum_{i\in I}u_i^{1/p}\cdot w_i^{1/q}=m_{1,1}$, 
%and $V^{1/p}\cdot W^{1/q}=\sum_{i\in I}v_i^{1/p}\cdot w_i^{1/q}=m_{2,1}$.
Indeed, the supremum in \eqref{eq:?} is no greater than that in \eqref{eq:??}; %here, we are saying 
at this point, we can only say ``no greater'' because the (following by \eqref{eq:u,v def}) expressions $x_i=(u_i/w_i)^{1/p}$ and $y_i=(v_i/w_i)^{1/p}$ of $x_i$ and $y_i$ in terms of $u_i,v_i,w_i$ will only be valid if $w_i\ne0$. 
%(actually, as can be easily disthe two suprema both equal. 

The important point here is that the set $\tT:=\tT_{I;m_{1,1},m_{1,p},m_{2,1},m_{2,p}}$ is compact, and the function $\tDe_p$ is continuous on it. So, $\tDe_p$ attains the (global) maximum on the set $\tT$ whenever $\tT\ne\emptyset$, which will be henceforth assumed wlog. 

For any vector $R=(r_i)_{i\in I}\in[0,\infty)^I$, let 
\begin{equation*}
	I_R:=\{i\in I\colon r_i>0\}.  %\quad	\1_R:=(\sign r_i)_{i\in I},\quad\0_R:=\1-\1_R, 
\end{equation*}
In view of \eqref{eq:U.W,V.W} and the condition (stated below \eqref{eq:?}) that $m_{1,1},m_{1,p},m_{2,1},m_{2,p}$ are strictly positive, for any $(U,V,W)\in\tT_{I;m_{1,1},m_{1,p},m_{2,1},m_{2,p}}$ we have  
\begin{equation}\label{eq:I cap I}
	I_U\cap I_W\ne\emptyset\quad\text{and}\quad I_V\cap I_W\ne\emptyset.   
\end{equation}

\smallskip 

\round{Further preparation for Lagrange multipliers \label{prep}}
%\noindent{\bf Round 5: Further preparation for Lagrange multipliers.}
%
Fix now any triple $(U^*,V^*,W^*)\in\tT_{I;m_{1,1},m_{1,p},m_{2,1},m_{2,p}}$ at which the maximum of $\tDe_p$ is attained. %For any vector $R=(r_i)_{i\in I}$, let $I_R:=\{i\in I\colon r_i>0\}$. 
Then clearly the triple $(U^*,V^*,W^*)$ is a maximizer of $\tDe_p$ over the set 
\begin{multline*}
	\tT^*_{I;m_{1,1},m_{1,p},m_{2,1},m_{2,p}} \\ 
	:=\{(U,V,W)\in\tT_{I;m_{1,1},m_{1,p},m_{2,1},m_{2,p}}\colon 
	I_U=I_{U^*}, I_V=I_{V^*}, I_W=I_{W^*}\}. 
\end{multline*}
Also, with the triple $(U^*,V^*,W^*)$ fixed, any triple 
$(U,V,W)\in\tT^*_{I;m_{1,1},m_{1,p},m_{2,1},m_{2,p}}$ may be identified with the triple $(U|_{I_{U^*}},V|_{I_{V^*}},W|_{I_{W^*}})$ of the restrictions of $U,V,W$ to the sets $I_U=I_{U^*}, I_V=I_{V^*}, I_W=I_{W^*}$, respectively; here, for instance, $U|_{I_{U^*}}=(u_i)_{i\in I_{U^*}}$; 
%, where the $u^*_i$'s %, $v^*_i$'s, $w^*_i$'s 
%are %of course 
%the coordinates of the vector $U^*%,V^*,W^*$, respectively
%$; 
so, $\tDe_p(U,V,W)$ may be considered a function of $(U|_{I_{U^*}},V|_{I_{V^*}},W|_{I_{W^*}})$.

\smallskip   

\round{Obtaining Lagrange multiplier equations \label{lagr eqs}}
%\noindent{\bf Round 6: Obtaining Lagrange multiplier equations.} 
Now we are ready to apply (say) the 
Carath\'eodory--John version of the Lagrange multiplier rule (see e.g.\ \cite[page~441]{pourciau}). In view of \eqref{eq:tDe2}, there exist some real numbers $\al,\la,\mu,\nu,\rho,\th$ (Lagrange multipliers) -- with $\al$ corresponding to the minimized $\tDe_p(U,V,W)$, and $\la,\mu,\nu,\rho,\th$ corresponding to the restrictions in \eqref{eq:U.W,V.W} and \eqref{eq:1.W,U.1,V.1} on $U^{1/p}\cdot W^{1/q}, V^{1/p}\cdot W^{1/q}$, \break
$U\cdot\1,V\cdot\1,\1\cdot W$, respectively -- 
such that 
\begin{equation}\label{eq:not 0}
	\al^2+\la^2+\mu^2+\nu^2+\rho^2+\th^2>0 
\end{equation}
and the triple $(U^*,V^*,W^*)$ is a solution 
to the following system of equations for $(U,V,W)$: 
\begin{alignat}{7}
& \forall\ i\in I_U\quad & \al(p-1)u_i^{-1/p}v_i^{1/p}&
&&=\la u_i^{-1/q}w_i^{1/q}&&	 &&+\nu, && && \label{(1)}\\	
& \forall\ i\in I_V\quad & \al u_i^{1/q}v_i^{-1/q}&
&&= &&\mu v_i^{-1/q}w_i^{1/q}	 && &&+\rho, && \label{(2)}\\	
& \forall\ i\in I_W\quad & 0&
&&=\la u_i^{1/p}w_i^{-1/p} &+&\mu v_i^{1/p}w_i^{-1/p}	 && && &&+%\frac\th{p-1}
\th. \label{(3)}
\end{alignat}
Multiplying (both sides of) equations \eqref{(1)} and \eqref{(2)} by $u_i$ and $v_i$, respectively, we have 
\begin{alignat}{6}
\al(p-1)u_i^{1/q}v_i^{1/p}
&=&\la u_i^{1/p}w_i^{1/q}&&	 &&+\nu u_i, && && \label{(4)}\\	
\al u_i^{1/q}v_i^{1/p}
&=& &&\mu v_i^{1/p}w_i^{1/q}	 && &&+\rho v_i && \label{(5)}
\end{alignat}
for all $i\in I$. 

A difficulty in analyzing these Lagrange multiplier equations is that some of the Lagrange multipliers $\al,\la,\mu,\nu,\rho,\th$ may take zero values. In a certain sense, this corresponds to the fact the difference between the left and right sides of inequality \eqref{eq:2} can attain its maximum (zero) value in a number of ways, including the cases when $X=Y$ and when %$X$ or 
$Y$ is a constant. %; there are also cases of quasi-maxima. 
Also, we have to account for cases when some of the values of $u_i,v_i,w_i$ are $0$, that is, when $i$ is not in the corresponding sets $I_U,I_V,I_W$. 

In particular, we have to consider the cases when $u_i>0$ or $v_i>0$ while $w_i=0$ \big(that is, when $i\in(I_U\cup I_V)\setminus I_W$\big). Recalling \eqref{eq:u,v def}, we see that, in terms of the ``original, pre-compactification'' variables $x_i,y_i,w_i$, these cases reflect the possibility for these variables to vary in such a way that for some $i\in I$ we have $w_i\downarrow0$ while $x_i\to\infty$ or $y_i\to\infty$ and, moreover, $x_i^p w_i$ or, respectively, $y_i^p w_i$ converges to a finite nonzero limit. This kind of phenomena may be thought of as part of the mass of the ``distribution`` of $U$ or $V$ running away to $\infty$. This brings us to the following round.

\smallskip  

\round{Analysis of Lagrange multipliers, part I: Removing ``the masses at $\infty$'' \label{mass infty}}
%\noindent{\bf Round 7: Analysis of Lagrange multipliers, part I: Removing ``the masses at $\infty$''.} 
\noindent Take any triple $(U,V,W)\in([0,\infty)^I)^3$ satisfying the Lagrange multiplier %system of 
equations \eqref{(1)}--\eqref{(3)}. On the set $I_W$, define the probability space by the condition $\P(\{i\})=w_i$ for all $i\in I_W$, and then define r.v.'s $X$ and $Y$ on this probability space by the conditions 
\begin{equation}\label{eq:X,Y}
\text{ $X(i)=x_i:=(u_i/w_i)^{1/p}$ and $Y(i)=y_i:=(v_i/w_i)^{1/p}$ for all $i\in I_W$. }	
\end{equation}
The r.v.'s $X$ and $Y$ are well defined, because $w_i>0$ for all $i\in I_W$ and $\sum_{i\in I_W}w_i=\sum_{i\in I}w_i=1$. Then, by \eqref{eq:tDe1} and \eqref{eq:De=..aBC},  
$\tDe_p(U,V,W)=\De_{p;A,B,C}(X,Y)$, 
%\begin{equation}\label{eq:De=..aBC}
%	\begin{aligned}
%	\tDe_p(U,V,W)=&A+\E X^{p-1}Y-\E^{p-1} X\,\E Y \\ 
%	&-\big(B+\E X^p-\E^p X\big)^{1/q}\,\big(C+\E Y^p-\E^p Y\big)^{1/p}, 
%\end{aligned}	
%\end{equation}
where 
\begin{equation}\label{eq:A,B,C}
	A:=\sum_{i\notin I_W}u_i^{1/q}v_i^{1/p},\quad B:=\sum_{i\notin I_W}u_i,\quad C:=\sum_{i\notin I_W}v_i,  
\end{equation}
``the masses at $\infty$''. 
By H\"older's inequality, here the condition $A\le B^{1/q}C^{1/p}$ in Lemma~\ref{lem:A,B,C} holds. 
So, $\tDe_p(U,V,W)\le\De_p(X,Y)$. 

Thus, it remains to show that $\De_p(X,Y)\le0$ for $X$ and $Y$ as in \eqref{eq:X,Y}, with $(U,V,W)\in([0,\infty)^I)^3$ satisfying the Lagrange multiplier %system of 
equations \eqref{(1)}--\eqref{(3)}.

\smallskip 

\round{Analysis of Lagrange multipliers, part II: Reduction to the case $Y=X+t$,\ \; $t\in\R$ \label{Y=X+t}}
%\noindent{\bf Round 8: Analysis of Lagrange multipliers, part II: Reduction to the case $Y=X+t$,\ \; $t\in\R$.} 
In terms of the $x_i$'s and $y_i$'s as in \eqref{eq:X,Y}, for $i\in I_W$ equations \eqref{(3)}, \eqref{(4)}, \eqref{(5)} can be rewritten as   
\begin{alignat}{6}
0
&=& \la x_i &&+\mu y_i	 && && && +%\frac\th{p-1}
\th, \label{(9)} \\ 
\al(p-1)x_i^{p-1}y_i
&=&\la x_i&&	 &&+\nu x_i^p, && && \label{(7)}\\	
\al x_i^{p-1}y_i
&=& && \mu y_i	 && && +\rho y_i^p. && \label{(8)}	
\end{alignat}

\begin{lemma}\label{lem:mu=0} 
Take any pair $(X,Y)\in([0,\infty)^{I_W})^2$ satisfying %the Lagrange multiplier system of 
equations \eqref{(9)}--\eqref{(8)} with $\mu=0$. Then $\De_p(X,Y)\le0$.  
\end{lemma}

\begin{proof} This proof consists in the consideration of a system of simple cases, keeping in mind the condition $\mu=0$. 
\begin{enumerate}[]
	\item \emph{Case 1: $\rho=0$.} 
	\begin{enumerate}[]
	\item \emph{Subcase 1.1: $\rho=0\ne\al$.} Then, by \eqref{(8)}, $x_i^{p-1}y_i=0$ for all $i\in I_W$. So, $\E X^{p-1}Y=0$, and inequality %\eqref{eq:le hat De} 
	$\De_p(X,Y)\le0$ obviously holds.  
	\item \emph{Subcase 1.2: $\rho=0=\al$.} Then, by \eqref{(7)}, $\la x_i+\nu x_i^p=0$ for all $i\in I_W$. 
	\begin{enumerate}[]
	\item \emph{Subsubcase 1.2.1: $\rho=0=\al$ and $\la=0=\nu$.} Then, by \eqref{(9)}, $\th=0$. So, we have a contradiction with \eqref{eq:not 0}. 
	\item \emph{Subsubcase 1.2.2: $\rho=0=\al$ and $\la\ne0$.} Then, by \eqref{(9)}, $x_i$ does not depend on $i\in I_W$; that is, the r.v.\ $X$ is a constant, and hence $\De_p(X,Y)=0$. 
	\item \emph{Subsubcase 1.2.2: $\rho=0=\al$ and $\la=0\ne\nu$.} Then, by \eqref{(7)}, $x_i=0$ for all on $i\in I_W$; that is, $X=0$, and hence $\De_p(X,Y)=0$. 
\end{enumerate}
\end{enumerate}
	\item \emph{Case 2: $\rho\ne0$.} 
	\begin{enumerate}[]
	\item \emph{Subcase 2.1: $\rho\ne0=\la$.} Then, by \eqref{(7)} and \eqref{(8)}, $\nu x_i^p=(p-1)\rho y_i^p$ for all $i\in I_W$. So, $Y=cX$ for some real $c\ge0$, and hence $\De_p(X,Y)=0$.  
	\item \emph{Subcase 2.2: $\rho\ne0\ne\la$.} Then, by \eqref{(9)}, $x_i$ does not depend on $i\in I_W$; that is, the r.v.\ $X$ is a constant, and hence $\De_p(X,Y)=0$.  
\end{enumerate}
\end{enumerate}
Thus, indeed in all cases we have $\De_p(X,Y)\le0$. 
\end{proof} 

So, by Lemma~\ref{lem:mu=0}, wlog $\mu\ne0$. So, in view of \eqref{(9)}, $Y=kX+t$ for some real $k$ and $t$. 

Now we need Chebyshev's integral inequality, which states that, if $f$ and $g$ are nondecreasing functions from $\R$ to $\R$, then for any r.v.\ $Z$ one has $\E f(Z)g(Z)\ge\E f(Z)\E g(Z)$ whenever all the three %integrals 
expectations here are finite; see e.g.\ Corollary~2 on page~318 in \cite{kemper77} (with $n=1$, $\phi=1$, and the probability distribution of $Z$ to play the role of the measure $\la$ there). This inequality follows immediately by taking the expectation of both sides of the obvious inequality $(f(Z)-f(Z_1))(g(Z)-g(Z_1))\ge0$, where $Z_1$ is an independent copy of $Z$. 

By Chebyshev's integral inequality and the mentioned log-convexity of $\E X^\al$ in $\al$, 
%H\"older's inequality (cf.\ \eqref{eq:lyapunov}), 
for $Y=kX+t$ with $k\le0$ we have 
$\E X^{p-1}Y\le\E X^{p-1}\,\E Y\le\E^{p-1} X\,\E Y$, which yields $\De_p(X,Y)\le0$, in view of \eqref{%eq:le hat De
eq:De_p}. So, wlog $k>0$, and then, because of the positive homogeneity of $\De_p(X,Y)$ in $Y$, wlog $Y=X+t$. 

\smallskip 

\round{Analysis of the case $Y=X+t$,\ \; $t\in\R$ \label{Y=X+t,analysis}}
%\noindent{\bf Round 9: Analysis of the case $Y=X+t$,\ \; $t\in\R$.}
Thus, to finish the proof of \eqref{eq:2}, 
it remains to prove 

\begin{lemma}\label{lem:X+t} 
For all real $t$ such that the r.v.\ $X+t$ is nonnegative, we have 
\begin{equation}\label{eq:de}
	\de(t):=\De_p(X,X+t)\le0. 
\end{equation} 
\end{lemma}

\begin{proof} 
In view of \eqref{%eq:le hat De
eq:De_p}, $\de(0)=0=\de'(0)$. So, it is enough to show that the function $\de$ is concave or, equivalently, that  the function $f$ given by the formula 
\begin{equation*}
	f(t):=\big(\E(X+t)^p-(\E X+t)^p\big)^{1/p}
\end{equation*}
for $t\in T:=\{s\in\R\colon X+s\ge0\}$ is convex. The set $T$ is an interval. So, it suffices to show that $f''(t)\ge0$ for all $t$ in the interior $\inter T$ of the set $T$ or, equivalently, that 
\begin{equation}\label{eq:H}
H:=%%%%p
(m_p-m_1^p)(m_{p-2}-m_1^{p-2})-(m_1^{p-1}-m_{p-1})^2\ge0, 	
\end{equation}
where 
\begin{equation*}
	m_r:=\E Y^r,
\end{equation*}
$Y=X+t$, and $t\in\inter T$. 
Here, by the positive homogeneity, for any fixed $t\in\inter T$ wlog 
\begin{equation*}
	m_1=\E Y=1. 
\end{equation*}

In principle, inequality \eqref{eq:H} can be proved by minimizing the $p$th moment $m_p$ of the r.v.\ $Y$ given the moments $m_{p-2},m_0=1,m_{p-1},m_1=1$ of $Y$ of orders $p-2,0,\break 
p-1,1$. Using results of, say, \cite{winkler88,ext-published}, we may assume that the support of the distribution of $Y$ consists of at most $\card\{p-2,0,p-1,1\}=4$ points, where $\card$ denotes the cardinality. This would reduce \eqref{eq:H} to a minimization problem involving $8$ variables (not counting $p$): $4$ variables for the points of the support of the distribution and $4$ variables for the corresponding masses. In our particular case, the minimization problem  can be further simplified by noticing that the moment functions mapping $x\in[0,\infty)$ to $x^{p-2},x^0,x^{p-1},x^1,x^p$ form a Tchebycheff--Markov system and hence we may assume that the support of the distribution of $Y$ consists of at most $2$ points; see e.g.\ \cite{karlin-studden} or \cite[Propositions~1 and 2] {pinelis2011tchebycheff}. Thus, we would have to deal with $4$ variables (not counting $p$): $2$ variables for the points of the support and $2$ variables for the masses. The values of the masses could be eliminated by solving the system of equations $m_0=1$ and $m_1=1$, which are linear with respect to the two masses. That would leave us with two variables, one for each of the two support points, plus another variable for $p$. 

Fortunately, again in our particular case, we can actually use a simple trick to reduce the problem to one involving just one variable in addition to $p$. Indeed, by the mentioned Lyapunov inequality (that is, the log-convexity of $m_r$ in $r$), 
%
%two simple applications of the mentioned Lyapunov inequality (that is, the log-convexity of $m_r$ in $r$) to reduce the problem to one involving just one variable in addition to $p$, as follows: 
%%
%%Using now the condition $1<p<2$ and the mentioned Lyapunov inequality (that is, the log-convexity of $m_r$ in $r$), we have 
$m_{p-1}\le m_1^{p-1}m_0^{2-p}=1$, $1=m_1\le m_{p-1}^{p-1}m_p^{2-p}$, and $1=m_0\le m_{p-2}^{p-1}m_{p-1}^{2-p}$, whence %$m_{p-2}\ge1$ and 
\begin{equation}\label{eq:>m_{p-1}>}
	1\ge m_{p-1}\ge m_*\vee m_{**},
\end{equation}
where 
\begin{equation*}
	m_*:=m_p^{-(2-p)/(p-1)}\quad\text{and}\quad m_{**}:=m_{p-2}^{-(p-1)/(2-p)}. 
%		,\quad 
%		\al:=\frac{2-p}{p-1}. 
\end{equation*}
Next, $m_*\ge m_{**}$ iff $m_p^{(2-p)^2}\le m_{p-2}^{(p-1)^2}$. So, by \eqref{eq:H} and \eqref{eq:>m_{p-1}>}, 
\begin{multline*}
	H\ge H_*:=%%%%p
	(m_p-1)\big(m_p^{(2-p)^2/(p-1)^2}-1\big)-(1-m_*)^2\\ 
	\text{if }m_p^{(2-p)^2}\le m_{p-2}^{(p-1)^2},
\end{multline*} 
\begin{multline*}	 
	H\ge H_{**}:=%%%%p
	\big(m_{p-2}^{(p-1)^2/(2-p)^2}-1\big)(m_{p-2}-1)-(1-m_{**})^2\\ 
	\text{if }m_p^{(2-p)^2}\ge m_{p-2}^{(p-1)^2}. 
\end{multline*}
Note that $H_*$ depends only on $p$ and $m_p$, whereas $H_{**}$ depends only on $p$ and $m_{p-2}$. 

It suffices to show that $H_*\ge0$ for all $p\in(1,2)$ and real $m_p\ge1$ and that $H_{**}\ge0$ for all $p\in(1,2)$ and real $m_{p-2}\ge1$. At this point, $m_p$ and $m_{p-2}$ may be considered free variables, with the only restriction that they take real values $\ge1$. Then, under the one-to-one correspondence between these free variables given by the formula $m_p^{(2-p)^2}\leftrightarrow m_{p-2}^{(p-1)^2}$, every value of $H_{**}$ turns into the corresponding value of $H_*$, and vice versa. 
%
%the one-to-one correspondence between these free variables given by the formula $m_p^{(2-p)^2}\leftrightarrow m_{p-2}^{(p-1)^2}$ shows that, for each $p\in(1,2)$, the set of values of $H_{**}$ over all $m_{p-2}\ge1$ is the same as the set of values of $H_*$ over all $m_p\ge1$. 
So, it is enough to show that $H_*\ge0$ for all $p\in(1,2)$ and real $m_p\ge1$. 

Making now the substitution $m_p=e^{(p-1)^2s}$, we can write 
\begin{multline*}
	\frac{H_*}{e^{2 (p-2) (p-1) s}}=h(s):=%h_p(s):=
	2 e^{(2-p) (p-1) s}-e^{(3-p) (p-1) s}-e^{(2-p) p s}+e^s-1. 
\end{multline*}
So, it suffices to show that $h(s)\ge0$ for all real $s\ge0$ (and all $p\in(1,2)$). Since $h$ is a linear combination of exponential functions, this can be done essentially algorithmically. Indeed, let 
$h_1(s):=h'(s)e^{(p-2) (p-1) s}$ and $h_2(s):=h_1'(s)e^{-(3 - 3 p + p^2) s}$. 
Then $h_2'(s)(2-p)^{-2} (p-1)^{-2}=p e^{-(p-1)^2 s}+(3-p) e^{-(2-p)^2 s}$, 
which is manifestly $>0$. So, $h_2$ is increasing (on the interval $[0,\infty)$), with $h_2(0)=0$. So, $h_2\ge0$ and hence $h_1$ is nondecreasing, with $h_1(0)=0$. So, $h_1\ge0$ and hence $h$ is nondecreasing, with $h(0)=0$. So, indeed $h\ge0$. 
Thus, Lemma~\ref{lem:X+t} is completely proved. 
\end{proof} 

This completes the proof of inequality \eqref{eq:2} and hence the proof of \eqref{eq:2nd} .
\end{proof}

Take now any $\thh\in[0,1]$. For real $t\ge0$, let 
	\begin{equation}\label{eq:g}
	g(t):=g_{\thh;X,Y}(t):=\EE_{p,\thh}(X+tY)-\EE_{p,\thh}(X)-t\EE_{p,\thh}(Y). 
\end{equation}
If $\EE_{p,\thh}(X+tY)=0$ for some $t\ge0$, then obviously $g(t)\le0$; otherwise \big(that is, if $\EE_{p,\thh}(X+tY)>0$\big), we can write 
\begin{equation}\label{eq:g'}
	g'(t)=\C_{p,\thh}(X+tY,Y)\EE_{p,\thh}(X+tY)^{1-p}-\EE_{p,\thh}(Y)\le0,  
\end{equation}
in view of already proved inequality \eqref{eq:2nd}; here, $g'(0)$ is understood as the right derivative of $g$ at $0$. So, for each real $t\ge0$ such that $g(t)>0$, we have $g'(t)\le0$. Also, $g(0)=0$ and the function $g$ is continuous. Suppose now that $g(1)>0$ and let $a:=\sup\{t\in[0,1]\colon g(t)=0\}$. Then $g(a)=0$ and $0\le a<1$; also, $g>0$ and hence $g'\le0$ on $(a,1]$. In view of  the mean value theorem, this contradicts the conditions $g(a)=0<g(1)$. Therefore, $g(1)\le0$; that is, inequality \eqref{eq:1st} holds. 

To finish the proof of Theorem~\ref{th:1}, it remains to show that inequalities \eqref{eq:1st} and \eqref{eq:2nd} are false in general if $p>2$ and $\thh\in(0,1]$. 
%This is easy to discern from the lines of the proof above. Indeed, 
To this end, suppose, e.g., that $\P(X=1)=\P(X=0)=1/2$. Let 
$\de_{p,\thh}(t):=\De_{p,\thh}(X,X+t)$; cf.\ \eqref{eq:de} and \eqref{eq:De_p,th}. 
Then 
%Then for the function $\de$ defined by \eqref{eq:de} we have 
$\de_{p,\thh}(0)=0=\de_{p,\thh}'(0+)$, whereas $\de_{p,\thh}''(0+)=(p-1) \thh^p/(2^p - 2 \thh^p)>0$, whence 
$\De_{p,\thh}(X,X+c)=\de_{p,\thh}(c)>0$ for small enough $c>0$. Take any such $c$ and let $Y:=X+c$, so that $\De_{p,\thh}(X,Y)>0$, that is, inequality \eqref{eq:2nd} is false. So, by \eqref{eq:g'}, $g'(0)>0$, which implies $g(t)>0$ for all small enough $t>0$. Thus, \eqref{eq:1st} with $tY$ in place of $Y$ is false if $t>0$ is small enough. \big(One might note that here $\de_{p,\thh}''(0+)=-\infty<0$ if $1<p<2$ and $\de_{2,\thh}''(0)=-(1-\thh^2)/(2-\thh^2)\le0$.\big)
%
%
% and $Y=X+c$ for some real $c\ge0$ (cf.\ Round~\ref{Y=X+t,analysis} of the above proof of \eqref{eq:2nd}). Then for the function $\de$ defined by \eqref{eq:de} we have $\de(0)=0=\de'(0)$, whereas $\de''(0)=(p-1) \thh^p/(2^p - 2 \thh^p)>0$, whence 
%$\De_p(X,Y)=\De_p(X,X+c)=\de(c)>0$ for small enough $c$, so that inequality \eqref{eq:2nd} is false. Therefore, %in view of \eqref{eq:4}, 
%for $g=g_{\thh;X,Y}$ as in \eqref{eq:g} we have 
%$g'(0)>0$, which implies $g(t)>0$ for all small enough $t>0$, that is, \eqref{eq:1st} with $tY$ in place of $Y$ is false if $t>0$ is small enough. 

The entire proof of Theorem~\ref{th:1} is now complete. 

\begin{remark}
The simple deduction of \eqref{eq:1st} from \eqref{eq:2nd} %at the end of the proof of Theorem~\ref{th:1}
in the paragraph containing formulas \eqref{eq:g} and \eqref{eq:g'} is essentially reversible, so that, vice versa, \eqref{eq:2nd} is easy to deduce from \eqref{eq:1st}. 
Indeed, 
take again any $\thh\in[0,1]$. If $\EE_{p,\thh}(X)=0$, then 
$\P(X=a)=1$ for some real constant $a\ge0$; moreover, if, in addition, $\thh<1$, then
necessarily $a=0$. So, inequality \eqref{eq:2nd} is trivial if $\EE_{p,\thh}(X)=0$. Therefore, wlog $\EE_{p,\thh}(X)>0$ and hence $g'(0)$ exists (cf.\ \eqref{eq:g'}), where $g$ is as in \eqref{eq:g}. Moreover, \eqref{eq:1st} with $tY$ in place of $Y$ yields $g(t)\le0$ for $t\ge0$. 
Since $g(0)=0$, we have $g'(0)\le0$. Now \eqref{eq:2nd} follows by the equality in \eqref{eq:g'}. \qed
\end{remark}

Inequality \eqref{eq:1} was conjectured in \cite{mink}. 

\bibliographystyle{abbrv}
%%%%\bibliographystyle{ims}
%%%%\bibliography{are.citations}
%%%%\bibliography{citat}
%%%
%%%%\bibliography{citations}
%%%

\bibliography{P:/pCloudSync/mtu_pCloud_02-02-17/bib_files/citations12.13.12}

% Preamble: \pgfplotsset{width=7cm,compat=1.12}

%\mathtoolsset{showonlyrefs,showmanualtags}

%% \appendix

%%  \bibliography{<your bibdatabase>}

%\bibliographystyle{abbrv}
%%\bibliographystyle{amsplain}
%
%\bibliography{P:/pCloudSync/mtu_pCloud_02-02-17/bib_files/citations12.13.12}

%\bibliography{P:/mtu_pCloud_02-02-17/bib_files/citations12.13.12}
%\bibliography{P:/mtu_pCloud_02-02-17/bib_files/citations12.13.12}
%\bibliography{C:/Users/ipinelis/Sync/mtu_Sync_01-19-17/bib_files/citations12.13.12}
%\bibliography{C:/Users/iosif/Sync/mtu_Sync_01-19-17/bib_files/citations12.13.12}
%\bibliography{C:/Users/Iosif/Dropbox/mtu/bib_files/citations12.13.12}

%% else use the following coding to input the bibitems directly in the
%% TeX file.

%%%\begin{thebibliography}{00}
%%%
%%%%% \bibitem{label}
%%%%% Text of bibliographic item
%%%
%%%\bibitem{}
%%%
%%%\end{thebibliography}
\end{document}